\documentclass[a4paper, 12pt, preprint, 3p, times]{P-article}
\usepackage{amssymb}
\usepackage{epsfig}
\usepackage{amsfonts}
\usepackage{amsmath, array}
\usepackage{euscript}
\usepackage{amscd}
\usepackage{amsthm}
\usepackage{ulem}
\usepackage{xcolor}
\usepackage{yfonts, mathrsfs, eufrak}
\usepackage{yfonts}

\usepackage[english]{babel}
\usepackage{blindtext}
\usepackage{mathtools} 

\usepackage[pdftex,
pdfauthor=Prosenjit Das,
pdftitle={Structure of A2-fibrations having fixed point free locally nilpotent derivations},
pdfsubject={Affine Algebraic Geometry, Commutative Algebra},
pdfkeywords={affine fibration, locally nilpotent derivation, retraction, fixed point free, kernel, residual variables},
pdfproducer=TeXStudio,
pdfcreator=pdflatex]{hyperref}

\usepackage{hyperref}

\hypersetup{
	colorlinks=true,
}

\DeclareMathAlphabet{\mathpzc}{OT1}{pzc}{m}{it}
\newtheorem{thm}{Theorem}[section]
\newtheorem{lem}[thm]{Lemma}
\newtheorem{prop}[thm]{Proposition}
\newtheorem{prob}[thm]{Problem}
\newtheorem{cor}[thm]{Corollary}

\newdefinition{defn}[thm]{Definition}
\newdefinition{ex}[thm]{Example}
\newdefinition{rem}[thm]{Remark}
\newdefinition{note}{Note}

\newcommand{\comment}[1]{}

\newcommand{\A}[1]{\mathbb{A}^{#1}}
\newcommand{\KerD}{\text{Ker}(D)}

\newcommand{\Der}[2]{\text{Der}_{#1}(#2)}
\newcommand{\Hom}[3]{\text{Hom}_{#1}(#2, #3)}
\newcommand{\Ker}[1]{\text{Ker}(#1)}
\newcommand{\nil}[1]{\text{nil}(#1)}
\newcommand{\Sym}[2]{\text{Sym}_{#1}(#2)}

\newcommand{\ol}[1] {\overline{#1}}
\newcommand{\ul}[1] {\underline{#1}}

\newcommand{\Spec}[1]{\text{Spec}(#1)}
\newcommand{\mbbQ}{\mathbb{Q}}

\newcommand{\mbbC}{\mathbb{C}}
\newcommand{\mbbN}{\mathbb{N}}
\newcommand{\trdeg}[2]{\text{tr.deg}_{#1}(#2)}

\begin{document}
\begin{frontmatter}
	\title{Structure of $\mathbb{A}^2$-fibrations having fixed point free locally nilpotent derivations}
		
	\author{Janaki Raman Babu}
	\address{Department of Mathematics, Indian Institute of Space Science and Technology, \\
		Valiamala P.O., Trivandrum 695 547, India\\
		email: \texttt{raman.janaki93@gmail.com, janakiramanb.16@res.iist.ac.in}}
	
	\author{Prosenjit Das\footnote{Corresponding author.}}
	\address{Department of Mathematics, Indian Institute of Space Science and Technology, \\
		Valiamala P.O., Trivandrum 695 547, India\\
		email: \texttt{prosenjit.das@gmail.com, prosenjit.das@iist.ac.in}}		
\begin{abstract}
	In this article we show that a fixed point free locally nilpotent derivation of an  $\mathbb{A}^2$-fibration over a Noetherian ring containing $\mathbb{Q}$ has slice.
	\end{abstract}
	\end{frontmatter}
	
	\section{Introduction} \label{Sec_Intro}
	Throughout this article rings will be commutative with unity. Let $R$ be a ring. For a prime ideal $P$ of $R$, let $k(P)$ denote the \textit{residue field} $R_P/PR_P$. The \textit{polynomial ring} in $n$ variables over $R$ is denoted by $R^{[n]}$. Let $A$ be an $R$-algebra. We shall use the notation $A = R^{[n]}$ to mean that $A$ is isomorphic, as an $R$-algebra, to a polynomial ring in $n$ variables over $R$. $A$ is called an \textit{$\A{n}$-fibration} or \textit{affine $n$-fibration} over $R$, if $A$ is finitely generated and flat over $R$, and $A \otimes_R k(P) = k(P)^{[n]}$ for all $P \in \Spec{R}$. $A$ is called a \textit{stably polynomial} algebra over $R$, if $A^{[m]} = R^{[n]}$ for some $m,n \in \mbbN$. Let $D: A \longrightarrow A$ be an $R$-derivation. $D$ is called \textit{irreducible}, if there does not exists $\alpha \in A\backslash A^*$ such that $D(A) \subseteq \alpha A$. $D$ is defined to be \textit{fixed point free} if $D(A)A = A$. $D$ is said to have a \textit{slice} $s \in A$, if $D(s) =1$. $D$ is called a \textit{locally nilpotent $R$-derivation} ($R$-LND), if for each $x \in A$, there exists $n \in \mbbN$ such that $D^n(x) =0$. Suppose that $\mathbb{Q} \hookrightarrow R$ and $D: A \longrightarrow A$ an $R$-LND. Then it is well known (\textit{slice theorem}) that $A = \KerD[s] = \KerD^{[1]}$, if $D$ has a slice $s \in A$; and the converse holds when $D$ is irreducible (see \cite[Proposition 2.1]{Wright_On-The-Jacobian-Conjecture}).
	
	\medskip
	
	A problem in affine algebraic geometry asks the following.
	
	\begin{prob} \label{Prob_A2_FPF-LND}
	Let $R$ be a ring containing $\mathbb{Q}$, $A$ an $\A{2}$-fibration over $R$, and $D: A \longrightarrow A$ a fixed point free $R$-LND. Does $D$ have a slice? 
	\end{prob}

The answer to Problem \ref{Prob_A2_FPF-LND} is affirmative for the case the $\A{2}$-fibration $A$ is trivial, i.e., $A = R^{[2]}$. When $R$ is a field, the result follows from the work of Rentschler (\cite{Rentchler_Operations-du-Groupe}), whereas the case $R$ is a UFD is proved by Daigle-Freudenburg in  \cite{Daigle-Freudenburg_UFD-LND-Rank-2}. Bhatwadekar-Dutta, in \cite{BD_LND}, established the result when $R$ is a Noetherian domain. The case $R$ is a general ring, under the assumption the LND has divergence  zero, was done by Berson-van den Essen-Maubach in \cite{Essen-Maubach-Berson_Der-div-zero}; and the most general case, which is stated below, was proved by van den Essen in \cite{Essen_Around-Cancellation}.

	\begin{thm} \label{Essen_FixdPtLND-R^[2]-trivKer}
	Let $R$ be a ring containing $\mathbb{Q}$ and $D$ a fixed point free $R$-LND of $A = R^{[2]}$. Then, $\KerD = R^{[1]}$ and $D$ has a slice, i.e., $A = \KerD^{[1]}$.
\end{thm} 

When the $\A{2}$-fibration $A$ is not known to be non-trivial over $R$, Problem \ref{Prob_A2_FPF-LND} has affirmative answers due to the results of Kahoui-Ouali (\cite{Kahoui_Fixd-pt-fr}, \cite{Kahoui_A2-fib_triviality-criterion}) under the assumption that $A$ is stably polynomial over $R$. In \cite{Kahoui_Fixd-pt-fr}, they established that if $R$ is normal, $A$ is stably polynomial over $R$ and $\KerD$ is an $\A{1}$-fibration over $R$, then $D$ has a slice, i.e., in particular, it follows that if ``$R$ is a Noetherian UFD and $A$ is a stably polynomial algebra over $R$'' (or simply, ``$R$ is a regular domain''), then $D$ has a slice \footnote{In this context, the following needs a mention. Freudenburg, in \cite{Freudenburg_Rxyz-slice}, showed that (\cite[Theorem 3.1]{Freudenburg_Rxyz-slice}) if $A$ is an $\A{2}$-fibration over $R = k^{[n]}$ where $k$ is a field containing $\mathbb{Q}$ and $n \ge 1$, then $A = R^{[2]}$ if and only if $A$ has an $R$-LND with slice (also see \cite[Corollary 2.3]{Daigle-Freudenburg_FamiliesAffFib}); and thereby raised a question (\cite[Question 2, pg. 3084]{Freudenburg_Rxyz-slice}) asking whether the same conclusion holds if the hypothesis ``$A$ has an $R$-LND with slice'' is replaced by ``$A$ has a fixed point free $R$-LND''. Kahoui-Ouali, in \cite{Kahoui_Fixd-pt-fr}, settled this question affirmatively. However, the above problem by Freudenburg essentially brought attention towards the question ``whether fixed point free LNDs of $\mathbb{A}^2$-fibrations are having slice''.}. Recently, in \cite{Kahoui_A2-fib_triviality-criterion}, they proved that the conclusion holds in more generality, specifically (\cite[Thorem 3.1, Theorem 2.4 \& Corollary 3.2]{Kahoui_A2-fib_triviality-criterion}),

\begin{thm} \label{Kahoui-A2-fib_trivial_criterion}
	Let $R$ be a ring containing $\mbbQ$, $A$ an $\A{2}$-fibration over $R$ and $D: A \longrightarrow A$ a fixed point free $R$-LND. If $A$ is stably polynomial over $R$, then $\KerD = R^{[1]}$ and $D$ has a slice, i.e., $A = \KerD^{[1]}$. Further, if $R$ is Noetherian and $A$ is locally stably polynomial over $R$, then $\KerD = \Sym{R}{N}$ for some rank one projective $R$-module $N$ and $A = \KerD^{[1]}$, i.e., $D$ has a slice.
\end{thm}
		
However, it is not known whether Problem \ref{Prob_A2_FPF-LND} has affirmative answer when the $\A{2}$-fibration $A$ is either non-trivial or not stably polynomial over $R$. Note that the above results of Kahoui-Ouali essentially prove that the $\A{2}$-fibration $A$ is trivial.

\medskip

In this article we show, without any extra hypothesis on $A$, that Problem \ref{Prob_A2_FPF-LND} has affirmative answer when the base ring $R$ is Noetherian, specifically (see Theorem \ref{Thm_A2-fib_Triv_FPF-LND}),

\smallskip

\noindent
{\bf Theorem A:} Let $R$ be a Noetherian ring containing $\mbbQ$ and $A$ an $\A{2}$-fibration over $R$ with a fixed point free $R$-LND $D: A \longrightarrow A$. Then, $\KerD$ is an $\A{1}$-fibration over $R$ and $D$ has a slice, i.e., $A = \KerD^{[1]}$. In particular, if $R$ is a normal domain, then $A = {\Sym{R}{I}}^{[1]}$ for some invertible ideal $I$ of $R$.

\medskip

In view of Theorem A, we investigate the hypothesis ``$A$ is stably polynomial over $R$'' in Theorem \ref{Kahoui-A2-fib_trivial_criterion} and find that it helps producing another fixed point free $R$-LND of $A$ for which the $\A{2}$-fibration $A$ becomes trivial. More precisely, we observe the following (see Theorem \ref{Thm_FPF-LND_implies_another-LND}).

\smallskip

\noindent
\textbf{Theorem B:}
		Let $R$ be a Noetherian domain containing $\mbbQ$ and $A$ an $\A{2}$-fibration over $R$ having a fixed point free $R$-LND. Then, $A$ has another irreducible $R$-LND $D: A \longrightarrow A$ such that $\KerD = R^{[1]}$, and $A$ is an $\A{1}$-fibration over $\KerD$. Further, the following are equivalent.
\begin{enumerate}
	\item [\rm (I)] $D$ is fixed point free.
	\item [\rm (II)] $A$ is stably polynomial over $R$.
	\item [\rm (III)] $A = R^{[2]}$.
\end{enumerate}

The following is the outline of this article. In Section \ref{Sec_Prelim}	we recall definition and preliminary results for subsequent use; in Section \ref{Sec_A1-Fib_and_FPF-LNDs} we review results on $\A{1}$-fibrations having fixed point free derivations; in Section \ref{Sec_Main-Results} we discuss $\A{2}$-fibrations having fixed point free LNDs; and in Section \ref{Sec_Examples} we discuss a few examples.     

\section{Preliminaries} \label{Sec_Prelim}
In this section we setup notations, recall definitions and quote some results.

\subsection*{\textbf{Notation:}}
	
	Given a ring $R$ and an $R$-algebra $A$ we fix the following notation.\\
	$
	\begin{array}{lll}
	R^*& : & \text{Group of units of $R$}.\\
	\nil{R} & : & \text{Nilradical of $R$}.\\ 
	K & : & \text{Total quotient ring of $R$}.\\
	\text{Pic}(R) & : & \text{Picard group of $R$}.\\
	\Sym{R}{M} & : & \text{Symmetric algebra of an $R$-module $M$}.\\
	\Omega_R(A) & : & \text{Universal module of $R$-differentials of $A$}.\\
	\Der{R}{A} & : & \text{Module of $R$-derivations of $A$}.\\
	\trdeg{R}{A} & : & \text{Transcendence degree of $A$ over $R$, where $R \subseteq A$ are domains}.\\
	A_P & : & A \otimes_{R} R_P, \ \text{for} \ P \in \Spec{R}.
	\end{array}
	$
	
	\subsection*{\textbf{Definitions:}}
	A reduced ring $R$ is called \textit{seminormal} if whenever $a^2 = b^3$ for some $a,b \in R$, then there exists $t \in R$ such that $t^3 = a$ and $t^2 =b$.
	
	\medskip
	A subring $R$ of a ring $A$ is called a \textit{retract} of $A$, if there exists a ring homomorphism $\phi: A \longrightarrow R$ such that $\phi(r) =r$ for all $r \in R$. 
	
	\medskip
	A subring $R$ of a ring $A$ is said to be \textit{inert (factorially closed)} in $A$, if $fg \in R$ implies $f, g \in R$ for all $f,g \in A \backslash \{ 0 \}$.
	
	\medskip
	Let $A$ be an $\A{n}$-fibration over a ring $R$. An $m$-tuple of elements $\ul{W}: = (W_1, W_2, \cdots, W_m)$ in $A$ which are algebraically independent over $R$ is called an $m$-\textit{tuple residual variable} of $A$ if $A \otimes_R k(P) = (R[\ul{W}] \otimes_R k(P))^{[n-m]}$ for all $P \in \Spec{R}$. 
	\subsection*{\textbf{Preliminary results:}}
	We now quote few results for later use. The first one is by Hamann (\cite[Theorem 2.8]{Haman_Invariance}).
	\begin{thm} \label{Hamann}
		Let $R$ be a Noetherian ring containing $\mathbb{Q}$ and $A$ an $R$-algebra such that $A^{[m]} = R^{[m+1]}$ for some $m \in \mathbb{N}$. Then, $A = R^{[1]}$.
	\end{thm}
	
	A classification of locally polynomial algebras by Bass-Connell-Wright (\cite[Theorem 4.4]{BCR_Local-Poly}) states
	
	\begin{thm}\label{BCW_Sym}
		Let $A$ be a finitely presented $R$-algebra such that $A_P$ is $R_P$-isomorphic to the symmetric algebra of some $R_P$-module for each $P \in \Spec{R}$. Then, $A$ is $R$-isomorphic to the symmetric algebra $\Sym{R}{M}$ for some finitely presented $R$-module $M$.
	\end{thm}
	
	The next result is by Swan (\cite[Theorem 6.1]{Swan_On-Sem}).
	
	\begin{thm} \label{Swan_SemiNormality}
		Let $R$ be a seminormal ring. Then, $\text{Pic}(R) = \text{Pic}(R^{[n]})$ for all $n \in \mathbb{N}$.
	\end{thm}
	
	Asanuma established the following structure theorem (\cite[Theorem 3.4]{Asanuma_fibre_ring}) of affine fibrations over Noetherian rings.
	\begin{thm} \label{Asanuma_struct-fib-th}
		Let $R$ be a Noetherian ring and $A$ an $\A{r}$-fibration over $R$. Then, $\Omega_R(A)$ is a projective $A$-module of rank $r$ and $A$ is an $R$-subalgebra (up to an isomorphism) of a polynomial ring $R^{[m]}$ for some $m \in \mathbb{N}$ such that $A^{[m]} = \mbox{Sym}_{R^{[m]}} (\Omega_R(A) \otimes_A R^{[m]})$ as $R$-algebras; and therefore, $A$ is retract of $R^{[n]}$ for some $n \in \mbbN$.
	\end{thm}
	
	\begin{cor} \label{Cor_A1-fib_Triv_DiffMod_Extended}
		
		Let $R$ be a Noetherian ring containing $\mbbQ$ and $A$ an $\A{1}$-fibration over $R$. If $\Omega_R(A)$ is extended from $R$, specifically, when $R$ is seminormal, then $A = \Sym{R}{N}$ for some finitely generated rank one projective $R$-module $N$.
	\end{cor}
	\begin{proof}
		Follows from Theorem \ref{Asanuma_struct-fib-th}, Theorem \ref{Swan_SemiNormality}, Theorem \ref{Hamann} and Theorem \ref{BCW_Sym}.
	\end{proof}

	We end this section by registering the following result of Das-Dutta (\cite[Corollary 3.6, Lemma 3.12, Theorem 3.16 \& Corollary 3.19]{DD_residual}).
	
	\begin{thm}\label{DD_fib}
		Let $R$ be a Noetherian ring and $A$ an $\A{n}$-fibration over $R$. Suppose, $\ul{W} \in A$ is an $m$-tuple residual variable of $A$. Then, $A$ is an $\A{n-m}$-fibration over $R[\ul{W}]$ and $\Omega_R(A) = \Omega_{R[\ul{W}]}(A) \oplus A^m$. Further, if $A$ is stably polynomial over $R \hookleftarrow \mbbQ$ and $n-m=1$, then $A = R[\ul{W}]^{[1]} = R^{[n]}$.
	\end{thm}
	
	It is to be noted that though Das-Dutta, in \cite{DD_residual}, proved Theorem \ref{DD_fib} (see \cite[Corollary 3.19]{DD_residual}) with the hypothesis that the base ring is a Noetherian domain containing $\mathbb{Q}$, from their proof it follows that Theorem \ref{DD_fib} holds over Noetherian rings (not necessarily domains) containing $\mathbb{Q}$.
	
	\section{Structure of $\A{1}$-fibrations having fixed point free derivations} \label{Sec_A1-Fib_and_FPF-LNDs}
	In \cite{Kahoui_Cancellation}, Kahoui-Ouali, proved (\cite[Corollary 2.5]{Kahoui_Cancellation}) that an $\A{1}$-fibration over a Noetherian domain containing $\mathbb{Q}$ is trivial if and only if it has a fixed point free derivation. In this section we show that the result of Kahoui-Ouali holds even over Noetherian rings (not necessarily domains) containing $\mathbb{Q}$ (see Proposition \ref{Kahoui_fpf-LND_A1-fib}), which we shall use in the next section. Though the proof of this observation follows from the proof of \cite[Corollary 2.5]{Kahoui_Cancellation}, for the convenience of the readers it is detailed here. First, we note a criterion for an algebra to be finitely generated, which also estimates an upper bound to the minimum number of generators of the algebra.
	 
	\begin{lem} \label{Lem_n-generated_modulo-nilR}
	Let $R$ be a ring, $\eta = \nil{R}$ and $A$ an $R$-algebra. Suppose, $x_i \in A$, $i=1,2, \cdots, n$ are such that $A/\eta A = R/\eta [\bar{x_1}, \bar{x_2}, \cdots, \bar{x_n}]$ where $\bar{x_i}$'s are the images of $x_i$'s in $A/\eta A$. Then, $A = R[x_1, x_2, \cdots, x_n]$.
	\end{lem}
	\begin{proof}
		Clearly, $A = R[x_1, x_2, \cdots, x_n] + \eta A$. Since there exists $\ell \in \mathbb{N}$ such that $\eta^{\ell} = (0)$, we see that $A = R[x_1, x_2, \cdots, x_n]$. 
	\end{proof}
	
	The following result gives a criterion for a singly-generated algebra to be a polynomial algebra.
	\begin{lem} \label{Lem_SinglyGeneratedAlg-PolyAlg}
		Let $R$ be a ring containing $\mathbb{Q}$, $\eta = \nil{R}$ and $A$ an $R$-algebra such that $A/\eta A$ is a singly-generated $R/\eta$-algebra and $(A/\eta A)^* = (R/\eta)^*$. Then, $A = R^{[1]}$ if and only if there exists $D \in \Der{R}{A}$ such that the induced $R/\eta$-derivation $\ol{D}: A/\eta A \longrightarrow A /\eta A$ is fixed point free.
	\end{lem}  
	\begin{proof}
		If $A = R^{[1]}$, then it is easy to see that there exists a fixed point free $R/\eta$-derivation of $A/\eta A$. So, we prove the converse. According to the hypotheses, $A/\eta A$ is generated by a single element over $R/\eta$, and therefore, by Lemma \ref{Lem_n-generated_modulo-nilR}, we have $A =R[x]$ for some $x \in A$. Suppose, $D \in \Der{R}{A}$ is such that the induced $R/\eta$-derivation $\ol{D}: A/\eta A \longrightarrow A/\eta A$ is fixed point free. Letting $\bar{x}$ as the image of $x$ in $A/\eta A$, it is easy to see that $\ol{D}(\bar{x}) \in (A/\eta A)^* = (R/\eta)^*$, i.e., the image of $D(x)$ in $A/\eta A$ is a unit in $R/\eta$, and therefore, $D(x)$ is a unit in $R$. We claim that $x$ does not satisfy any algebraic relation over $R$. On the contrary, if there exists $a_i \in R$ for all $i =0, 1, \cdots, n$, $a_n \ne 0$ such that $a_0 + a_1 x+ \cdots + a_n x^n =0$, then we have $D(x) (a_1 + 2a_2 x + 3a_3 x^2  \cdots + na_n x^{n-1}) =0$. Since $D(x) \in R^*$, we get $a_1 + 2a_2 x + 3a_3 x^2  \cdots + na_n x^{n-1} =0$, from which again we see that $D(x)(2a_2 + 6a_3 x + 12a_4 x^2 + \cdots + n(n-1)a_n x^{n-2}) =0$. Repeating same arguments, we eventually get $D(x)(n!)a_n =0$, i.e., $a_n =0$ which is a contradiction to our assumption that $a_n \ne 0$. This proves that $A = R[x] = R^{[1]}$.  
	\end{proof}
	
	We now prove the main result of this section: an $\A{1}$-fibration over a Noetherian ring containing $\mathbb{Q}$ is trivial if and only if it has a fixed point free derivation.
	
	\begin{prop} \label{Kahoui_fpf-LND_A1-fib}
		Let $R$ be a Noetherian ring containing $\mbbQ$ and $A$ an $\A{1}$-fibration over $R$. Then the following are equivalent.
		
		\begin{enumerate}
			\item [\rm (I)] $A = R^{[1]}$.
			\item [\rm (II)] $\Omega_R(A)$ is a free $A$-module.
			\item [\rm (III)] $\Omega_R(A)$ is a stably free $A$-module.
			\item [\rm (IV)] There exists $D \in \Der{R}{A}$ such that $D$ is fixed point free.
		\end{enumerate}
	\end{prop}
	\begin{proof}
		\uline{(I) $\implies$ (II)}, \uline{(II) $\implies$ (III)} and \uline{(I) $\implies$ (IV)}: Obvious.
		
		\medskip
		
		\uline{(III) $\implies$ (I)}: Suppose that $\Omega_R(A)$ is a stably free $A$-module. Then, there exists $n \in \mathbb{N}$ such that $\Omega_R(A) \oplus A^n = A^{n+1}$. Since $A$ is an $\A{1}$-fibration over $R$, by Theorem \ref{Asanuma_struct-fib-th}, there exists $\ell \in \mathbb{N}$ such that $A$ is an $R$-subalgebra of $B = R^{[\ell]}$ and $A^{[\ell]} = \Sym{B}{\Omega_R(A) \otimes_A B}$, and therefore, we have $A^{[n+\ell]} = A \otimes_R B^{[n]} = A \otimes_R B \otimes_B B^{[n]} = A^{[\ell]} \otimes_B B^{[n]} = \Sym{B}{(\Omega_R(A) \otimes_A B} \otimes_B B^{[n]} =  \Sym{B}{(\Omega_R(A) \oplus A^n) \otimes_{A} B}$. Now, since $(\Omega_R(A) \oplus A^n) \otimes_{A} B = B^{n+1}$, we essentially have $A^{[n+\ell]} = B^{[n+1]} = R^{[n+\ell +1]}$, and hence, by Theorem \ref{Hamann}, we get $A = R^{[1]}$. 

		\medskip
			
		\uline{(IV) $\implies$ (I)}: Suppose that $D \in \Der{R}{A}$ is fixed point free. Let us assume that $R$ is reduced. Since the total quotient ring $K$ of $R$ is zero-dimensional reduced Noetherian ring, we see that $A \otimes_R K = K^{[1]} = K[U]$, say, for some $U \in A$. Suppose, $D_1 \in \Der{R}{A}$. Letting $D(U) = \alpha \in A$ and $D_1(U) = \beta \in A$, we have $\alpha D_1 = \beta D$. Since $D$ is fixed point free there exists $\alpha_1, \alpha_2, \cdots, \alpha_m \in A$ and $u_1, u_2, \cdots, u_m \in A$ such that $\displaystyle \sum_{i=1}^{m} \alpha_i D(u_i) =1$, and therefore, $\displaystyle \sum_{i=1}^{m} \alpha_i \beta D(u_i) = \beta$. Now, since $\alpha D_1 = \beta D$, we get $\displaystyle \sum_{i=1}^{m} \alpha_i \alpha D_1(u_i) = \beta$, i.e., $\displaystyle \alpha \sum_{i=1}^{m} \alpha_i D_1(u_i) = \beta$. This shows that $\alpha D_1 = \displaystyle \alpha \sum_{i=1}^{m} \alpha_i D_1(u_i) D$.
		
		\medskip
		Let $\tilde{D}: A \otimes_R K \longrightarrow A \otimes_R K$ be the extension of $D$. Clearly, $\tilde{D}$ is fixed point free. Since $A \otimes_R K = K[U]$, we have $\tilde{D}(U) = D(U) = \alpha \in K^*$, i.e., $\alpha$ is a non-zero divisor in $R$. Since $A$ is flat over $R$, $\alpha$ remains a non-zero divisor in $A$, and therefore, $D_1 = \displaystyle \sum_{i=1}^{m} \alpha_i D_1(u_i) D$. This proves that $\Der{R}{A} = \Hom{A}{\Omega_R(A)}{A}$ is a free $A$-module of rank one. Since $\Omega_R(A)$ is a projective $A$-module, it is a reflexive $A$-module, and therefore, $\Omega_R (A)$ is a free $A$-module. Consequently, by ``(II) $\implies$ (I)'', we get $A = R^{[1]}$.
		
		\medskip
		Now, we suppose that $R$ is not reduced. Set $\eta := \text{Nil}(R)$. Clearly, the induced $R/\eta$-derivation $\overline{D}: A/\eta A \longrightarrow A/\eta A$ is fixed point free. Since $A/\eta A$ is an $\A{1}$-fibration over $R/\eta $, from the previous discussion we have $A/\eta A = (R/\eta)^{[1]} = (R/\eta)[X]$, say, and therefore, by Lemma \ref{Lem_n-generated_modulo-nilR}, we get $A = R[X]$. Finally, due to Lemma \ref{Lem_SinglyGeneratedAlg-PolyAlg}, it follows that $A = R[X] = R^{[1]}$.
	\end{proof}

	\section{Structure of $\A{2}$-fibrations having fixed point free LNDs} \label{Sec_Main-Results}
	First, we note the following lemmas. The first one is easy to prove. 
	\begin{lem} \label{Lem_Der-induced-by-Retraction}
		Let $C \subseteq B$ be algebras over a ring $R$ with a retraction $\phi: B \longrightarrow C$. Suppose $\tilde{D}: B \longrightarrow B$ is an $R$-derivation. Then, $D: = (\phi \circ \tilde{D})|_C: C \longrightarrow C$ is an $R$-derivation.
	\end{lem}
	
	\begin{lem} \label{Lem_A=C[W]_implies_W-Residual}
		Let $C$, $A$ be algebras over a Noetherian ring $R$ containing $\mbbQ$ such that $A$ is an $\A{2}$-fibration over $R$ and $A = C[W] = C^{[1]}$. Then, $C$ is an $\A{1}$-fibration over $R$ and $A$ is an $\A{1}$-fibration over $R[W]$.
	\end{lem}
	\begin{proof}
		Clearly, $C$ is a finitely generated $R$-subalgebra of $A$, and further, $C$, being a direct summand of the flat $R$-module $A$, is flat over $R$. Let $P \in \Spec{R}$. Now, $k(P)^{[2]} = A \otimes_R k(P) = (C \otimes_R k(P))^{[1]}$, and therefore, by Theorem \ref{Hamann}, we get $C \otimes_R k(P) = k(P)^{[1]}$. This shows that $C$ is an $\A{1}$-fibration over $R$. Again, as $C \otimes_R k(P) = k(P)^{[1]}$, we see that $A \otimes_R k(P) = (C \otimes_R k(P))[W] = (R[W] \otimes_R k(P))^{[1]}$. This proves that $W$ is a residual variable of $A$, and therefore, by Theorem \ref{DD_fib}, $A$ is an $\A{1}$-fibration over $R[W]$. This completes the proof.
	\end{proof}

		As a consequence of Proposition \ref{Kahoui_fpf-LND_A1-fib} we observe the following special case of Theorem A, as well as Theorem \ref{Kahoui-A2-fib_trivial_criterion}; the technique provides an independent short proof to it.
	
	\begin{prop}\label{Lem_fpf-LND-on-Sym_Ker_fg-flat}
		Let $R$ be a Noetherian ring containing $\mbbQ$ and $A = \Sym{R}{M}$ for some finitely generated rank two projective $R$-module $M$. Suppose, $D: A \longrightarrow A$ is a fixed point free R-$LND$, then $\KerD = \Sym{R}{N}$ for some finitely generated rank one projective $R$-module $N$ and $A = \KerD^{[1]}$.
	\end{prop}
	\begin{proof}
		By Theorem \ref{Essen_FixdPtLND-R^[2]-trivKer} we see that $A_P = \KerD_P^{[1]}$ and $\KerD_P = R_P^{[1]}$ for all $P \in \Spec{R}$. This shows that $A_Q = \KerD_Q^{[1]}$ for all $Q \in \Spec{\KerD}$, and therefore, by Theorem \ref{BCW_Sym}, we have $A = \Sym{\KerD}{L}$ for some finitely generated rank one projective $\KerD$-module $L$, which proves that $\KerD$, being a retract of the finitely generated $R$-algebra $A$, is a finitely generated $R$-subalgebra of $A$. Since $\KerD_P = {R_P}^{[1]}$ for all $P \in \Spec{R}$, by Theorem \ref{BCW_Sym}, $\KerD = \Sym{R}{N}$ for some rank one projective $R$-module $N$. Moreover, since $A = \Sym{\KerD}{L}$, by Proposition \ref{Kahoui_fpf-LND_A1-fib}, we see that $A =\KerD^{[1]}$.
	\end{proof}
	
	We now prove our main result (Theorem A).
	
	\begin{thm} \label{Thm_A2-fib_Triv_FPF-LND}
		Let $R$ be a Noetherian ring containing $\mbbQ$ and $A$ an $\A{2}$-fibration over $R$. Suppose, $D: A \longrightarrow A$ is a fixed point free $R$-LND. Then, $\KerD$ is an $\A{1}$-fibration over $R$ and $A = \KerD^{[1]}$. Further, if $\Omega_R(A)$ is extended from $R$, specifically, when $R$ is seminormal, then $\KerD = {\Sym{R}{N}}$ for some finitely generated rank one project $R$-module $N$.
	\end{thm}
	\begin{proof}
		Since $A$ is an $\A{2}$-fibration over $R$, by Theorem \ref{Asanuma_struct-fib-th}, there exists $B = R^{[n]}$ such that $A \subseteq B$ and $A^{[n]} = A \otimes_R B = \Sym{B}{\Omega_R(A) \otimes_A B}$ where $\Omega_R(A)$ is a finitely generated projective $A$-module of rank two. Let $\tilde{D} := D \otimes 1: A \otimes_R B \longrightarrow A \otimes_R B$ be the trivial extension of $D$. Note that $\tilde{D}$ is fixed point free and $\Ker{\tilde{D}} = \KerD \otimes_R B$. Since $\Omega_R(A)$ is a projective $A$-module, $\Omega_R(A) \otimes_A B$ is a projective $B$-module, and therefore, applying Proposition \ref{Lem_fpf-LND-on-Sym_Ker_fg-flat}, we get $\KerD \otimes_R B = \Sym{B}{L}$ for some finitely generated rank one projective $B$-module $L$ and $A \otimes_R B = (\KerD \otimes_R B)^{[1]}$. Since $B= R^{[n]}$, we have $A^{[n]} = A \otimes_R B = (\KerD \otimes_R B)^{[1]} = \KerD^{[n+1]}$, and therefore, by Theorem \ref{Hamann}, we have $A = \KerD^{[1]}$. Finally, using Lemma \ref{Lem_A=C[W]_implies_W-Residual} we see that $\KerD$ is an $\A{1}$-fibration over $R$. 
	
		\medskip
		
		Now, we assume that $\Omega_R(A)$ is extended from $R$, i.e., $\Omega_R(A) = M \otimes_R A$ for some $R$-module $M$. Since $\Omega_R(A)$ is a projective $A$-module of rank two and $A$ is faithfully flat over $R$, due to faithful descent property of finite projective module we see that $M$ is a rank two projective $R$-module. Since $A$ is an $\A{2}$-fibration over $R$, from earlier arguments we have $A \subseteq R^{[n]}$ and
		
		$$
		\begin{array}{ll}
		A^{[n]} & = \Sym{R^{[n]}}{\Omega_R(A) \otimes_{A} R^{[n]}}\\
		 ~ & = \Sym{R^{[n]}}{(M \otimes_{R} A) \otimes_A R^{[n]}}\\
		 ~ & = \Sym{R}{M}  \otimes_R R^{[n]}\\
		 ~ & = \Sym{R}{M}^{[n]}
		\end{array}
		$$
		
Thus, we have $\KerD^{[n+1]} = A^{[n]} = \Sym{R}{M}^{[n]}$, and therefore, for each $P \in \Spec{R}$, we get $\KerD_P^{[n+1]} = (\Sym{R}{M})_P = R_P^{[n+2]}$, from which, by Theorem \ref{Hamann}, we see that $\KerD_P = R_P^{[1]}$. Now, applying Theorem \ref{BCW_Sym}, we have $\KerD = \Sym{R}{N}$ for some rank one projective $R$-module $N$. When $R$ is seminormal, the result follows directly from Corollary \ref{Cor_A1-fib_Triv_DiffMod_Extended}, as $\KerD$ is an $\A{1}$-fibration over $R$.
	\end{proof}

As a consequence of Theorem \ref{Thm_A2-fib_Triv_FPF-LND} we observe the following characterization of $\A{2}$-fibrations that are polynomial algebras over a subalgebra of it.

\begin{cor} \label{Cor_Structure_A2-Fib_Equivalences}
	Let $R$ be a Noetherian ring containing $\mbbQ$ and $A$ an $\A{2}$-fibration over $R$. Then, the following statements are equivalent.
	\begin{enumerate}
		\item [\rm (I)] $A$ has a fixed point free $R$-LND.
		\item [\rm (II)] $A$ has an $R$-LND with slice.
		\item [\rm (III)] $A = C[W] = C^{[1]}$ for some $R$-subalgebra $C$ of $A$.
			
		\item [\rm (IV)] $A = C[W] = C^{[1]}$ where $C \subseteq A$ is an $\A{1}$-fibration over $R$.
		\item [\rm (V)]$A$ is an $\A{1}$-fibration over $R[W] = R^{[1]}$ where $W \in A$ and there exists $B = R[W][U_1, U_2, \cdots, U_n] = R[W]^{[n]}$, for some $n \in \mbbN$, along with a retraction $\phi: B \longrightarrow A$ such that $\partial_W (\phi(U_i)) =0$.
		\end{enumerate}
\end{cor}
	
	\begin{proof}
		\underline{(I) $\Longleftrightarrow$ (II) $\implies$ (III):} Follows from Theorem \ref{Thm_A2-fib_Triv_FPF-LND}.
		
		\medskip
		\underline{(III) $\implies$ (IV):} Follows from Lemma \ref{Lem_A=C[W]_implies_W-Residual} 
		\medskip
		
		\underline{(IV) $\implies$ (II):} Since $A = C[W] = C^{[1]}$, $A$ has a $C$-LND $D$ with a slice, and therefore, $D$ is an $R$-LND of $A$ with a slice.
		
		\medskip
		\underline{(IV) $\implies$ (V):} From Lemma \ref{Lem_A=C[W]_implies_W-Residual} it follows that $A$ is an $\A{1}$-fibration over $R[W]$. Now, since $C$ is an $\A{1}$-fibration over $R$, by Theorem \ref{Asanuma_struct-fib-th}, there exists $B' = R[U_1, U_2, \cdots, U_{n}] = R^{[n]}$ for some $n \in \mbbN$ along with a retraction $\phi_1: B' \longrightarrow C$, which induces a retraction $\phi: B \longrightarrow C[W] = A$ such that $\phi|_{B'} = \phi_1$ where $B = B'[W]$. Clearly, $\partial_W(\phi(U_i)) = 0$ for all $i = 1, 2, \cdots, n$.
		
		\medskip
		\underline{(V) $\implies$ (II):} 
		Set $D : = (\phi \circ \partial_W)|_A: A \longrightarrow A$. By Lemma \ref{Lem_Der-induced-by-Retraction} it follows that $D$ is an $R$-derivation of $A$. We shall show that $D$ is an $R$-LND with slice $W$. Clearly, $D(W) =1$.
		
		\medskip
		Let $\alpha(\ul{U}) \in A \cap R[U_1, U_2, \cdots, U_n]$. Note that $\phi(\alpha(\ul{U})) = \alpha(\ul{U})$. One may check that
		\begin{eqnarray}
		D(\alpha(\ul{U})) &=& 0\\
		\text{and} & ~ &\nonumber\\ D^i(\alpha(\ul{U})W^m) &=& m(m-1) \cdots (m-i+1)\ \alpha(\ul{U})W^{m-i} \ \ \text{for all $i =1, 2, \cdots, m$}
		\end{eqnarray}
		
		Let $f \in A$. Then, $f = \alpha_0(\ul{U}) + \alpha_1(\ul{U}) W + \alpha_2(\ul{U}) W^2 + \cdots + \alpha_m (\ul{U}) W^m$ for some $\alpha_i(\ul{U})$'s in $R[U_1, U_2, \cdots, U_n]$, and therefore, $f = \phi(f) = \phi(\alpha_0(\ul{U})) + \phi(\alpha_1(\ul{U})) W + \phi(\alpha_2(\ul{U})) W^2 + \cdots + \phi(\alpha_m (\ul{U})) W^m$. Now, using (1) and (2) we see that $D^{m+1}(f) =0$. This shows that $D$ is an $R$-LND of $A$ with slice $W$.
	\end{proof}
	
One should note that Corollary \ref{Cor_Structure_A2-Fib_Equivalences} is related to the below stated problem on the structure of $\A{2}$-fibration.

\begin{prob} \label{Prob_A2-Structure}
	Let $R$ be a ring containing $\mathbb{Q}$ and $A$ an $\A{2}$-fibration over $R$. Is then $A$ an $\A{1}$-fibration over $R[V]$ for some $V$ in $A$?
\end{prob}			 

To know the origin of Problem \ref{Prob_A2-Structure}, one may refer to \cite{Asan-Bhatw_Struct-A2-fib} (also see \cite{Dolgachev-Weisfeiler}, \cite{Sat_Pol-two-var-DVR}, \cite{Asanuma_fibre_ring} and \cite{BD_AFNFIB}) . 
While Problem \ref{Prob_A2-Structure} is open in general, it is known that it has negative answer even when $R$ is a two-dimensional regular factorial domain (see Example \ref{Ex_Hochster}). However, the following landmark results give partial affirmative answers to Problem \ref{Prob_A2-Structure}. Sathaye (\cite{Sat_Pol-two-var-DVR}) proved that $A = R^{[2]}$, if $R$ is a DVR. A result of Bass-Connell-Wright (\cite{BCR_Local-Poly}) along with the result of Sathaye show that $A = R^{[2]}$ holds even if $R$ is a PID. Later, Asanuma-Bhatwadekar (\cite{Asan-Bhatw_Struct-A2-fib}) showed that $A$ is an $\mathbb{A}^1$-fibration over $R[W]$ for some $W \in A$, if $R$ is an one-dimensional Noetherian ring. For more related results one may look at \cite{Derksen_cancellation},
\cite{Essen_Around-Cancellation}, \cite{Daigle-Freudenburg_FamiliesAffFib}, \cite{Freudenburg_Rxyz-slice}, \cite{Kahoui_Fixd-pt-fr},
\cite{Kahoui_Cancellation}, \cite{Das_cancel}, and \cite{Kahoui_A2-fib_triviality-criterion}.

\medskip

We now prove Theorem B.
	
	\begin{thm}\label{Thm_FPF-LND_implies_another-LND}
		Let $R$ be a Noetherian domain containing $\mbbQ$ and $A$ an $\A{2}$-fibration over $R$ having a fixed point free $R$-LND. Then, $A$ has another irreducible $R$-LND $D: A \longrightarrow A$ such that $\KerD = R^{[1]}$, and $A$ is an $\A{1}$-fibration over $\KerD$. Further, the following are equivalent.
		\begin{enumerate}
			\item [\rm (I)] $D$ is fixed point free.
			\item [\rm (II)] $A$ is stably polynomial over $R$.
			\item [\rm (III)] $A = R^{[2]}$.
		\end{enumerate}
	\end{thm}
	\begin{proof} Suppose, $\delta: A \longrightarrow A$ is a fixed point free $R$-LND. Then, by Theorem \ref{Thm_A2-fib_Triv_FPF-LND}, $\Ker{\delta}$ is an $\A{1}$-fibration over $R$ and $A = \Ker{\delta}[V]= \Ker{\delta}^{[1]}$ for some $V \in A$. Since $K$ is the quotient field of $R$, we have $\Ker{\delta} \otimes_R K = K[U_0] = K^{[1]}$ for some $U_0$ in $\Ker{\delta}$, and therefore, $A \otimes_R K = K[V, U_0]$. Since $\Ker{\delta}$ is finitely generated over $R$, there exists $t \in R \backslash \{ 0 \}$ such that $\Ker{\delta}[1/t] = R[1/t][U_0]$, which enables us to choose $\alpha \in \mbbN$ and a $K$-LND $\tilde{D}$ on $A \otimes_R K$ such that $\tilde{D}(V) = 0$, $\tilde{D}(U_0) = t^{\alpha}$, and $\tilde{D}(A) \subseteq A$. So, $D:= \tilde{D}|_A$ is an $R$-LND of $A$ such that $R[V] \subseteq \KerD$. Since $R$ is Noetherian, through proper reduction, we can ensure irreducibility of $D$. Now, since $A = \Ker{\delta}[V]$, by Lemma \ref{Lem_A=C[W]_implies_W-Residual}, $A$ is an $\A{1}$-fibration over $R[V]$. This shows that $R[V]$ is inert in $A$, and hence, it is algebraically closed in $A$. Note that $\KerD$ is also algebraically closed in $A$. Now, since $R[V] \subseteq \KerD$ and $\trdeg{R}{R[V]} = \trdeg{R}{\KerD}$, we have $\KerD$ is algebraic over $R[V]$, and therefore, $\KerD = R[V]$.
		
		\medskip
		We now prove the equivalence of (I), (II) and (III).
		
		\smallskip
		\underline{(I) $\Longleftrightarrow$ (III):} Follows from Proposition \ref{Kahoui_fpf-LND_A1-fib}.
		
		\medskip
		\underline{(III) $\implies$ (II):} Obvious. 
		
		\medskip
		\underline{(II) $\implies$ (I):} Since $A$ an $\A{1}$-fibration over $\KerD = R^{[1]}$, we see that $A \otimes_R k(P)$ is an $\A{1}$-fibration over $\KerD \otimes_R k(P) = k(P)^{[1]}$ for all $P\in \Spec{R}$, and therefore, by Corollary \ref{Cor_A1-fib_Triv_DiffMod_Extended}, we get $A \otimes_R k(P) = (\KerD \otimes_R k(P))^{[1]}$ for all $P \in \Spec{R}$. Since $A$ is stably polynomial over $R$, applying Theorem \ref{DD_fib}, we conclude the implication. 
	\end{proof}

	\begin{rem} \label{Rem_KO-FPF-LND} As a corollary of Theorem \ref{Thm_A2-fib_Triv_FPF-LND} we get Kahoui-Ouali's result on triviality of stably polynomial $\A{2}$-fibration having a fixed point free LND, i.e., Theorem \ref{Kahoui-A2-fib_trivial_criterion}.
	
	\begin{proof} Let $A^{[m]} = R^{[m+2]}$.
		Using a standard reduction technique (see \cite[Lemma 4.3]{Kahoui_A2-fib_triviality-criterion} for the details) we get a finitely generated $\mbbQ$-algebra $R_0$ which is a subring of $R$ and a finitely presented $R_0$ subalgebra $A_0$ of $A$ such that $A_0^{[m]} = R_0^{[m+2]}$, $A_0 \otimes_{R_0} R = A$, $D(A_0) \subseteq A_0$ and $D_0 : = D|_{A_0}$ is a fixed point free $R_0$-LND. Using Theorem \ref{Thm_A2-fib_Triv_FPF-LND} we get $A_0 = \Ker{D_0}^{[1]}$, and therefore, we have ${A_0}^{[m]} = \Ker{D_0}^{[m+1]} = {R_0}^{[m+2]}$, from which, by Theorem \ref{Hamann}, it follows that $\Ker{D_0} = {R_0}^{[1]}$. This shows that we have $A_0 = {R_0}^{[2]}$, and therefore, by the properties of $A_0$ and $R_0$ it follows that $A = R^{[2]}$. Now, on applying Theorem \ref{Essen_FixdPtLND-R^[2]-trivKer}, we conclude that $\KerD = R^{[1]}$ and $A = \KerD^{[1]}$. 
		
		\smallskip
		
		Next, we assume that $R$ is Noetherian and $A$ is a locally stably polynomial algebra over $R$. Since $D$ is a fixed point free $R$-LND of $A$, by Theorem \ref{Thm_A2-fib_Triv_FPF-LND}, we have $\KerD$ is an $\A{1}$-fibration over $R$ and $A = \KerD^{[1]}$. Since $A$ is locally stably polynomial over $R$ and $A = \KerD^{[1]}$, by Theorem \ref{Hamann}, we see that $\KerD$ is a locally polynomial algebra over $R$, and therefore, by Theorem \ref{BCW_Sym}, we get $\KerD = \Sym{R}{N}$ for some rank one projective $R$-module $N$. 
	\end{proof}
\end{rem}

	\begin{rem} \label{Rem_Cor-Th-Combined}
		From Corollary \ref{Cor_Structure_A2-Fib_Equivalences} and the proof of Theorem \ref{Thm_FPF-LND_implies_another-LND} we note the following.
		\begin{enumerate}
			\item [\rm A.] Let $R$ be a Noetherian domain containing $\mbbQ$, $A$ an $\A{2}$-fibration over $R$ and $\delta: A \longrightarrow A$ a fixed point free $R$-LND. Then, there exists $V \in A$ such that 
			\begin{enumerate}
				\item [\rm (I)] $\Ker{\delta}$ is an $\A{1}$-fibration over $R$ and $A = \Ker{\delta}[V] = \Ker{\delta}^{[1]}$. Further, $A$ is stably polynomial algebra over $R$ if and only if $\Ker{\delta} = R^{[1]}$, i.e, $A = R^{[2]}$.
				
				\item [\rm (II)] $A$ has another irreducible $R$-LND $D$, not necessarily fixed point free, such that $\Ker{D} = R[V] = R^{[1]}$ and $A$ is an $\A{1}$-fibration over $R[V]$. Further, $A = R[V]^{[1]}$ if and only if $D$ is fixed point free.
			\end{enumerate}
		
		\item [\rm B.] Let $R$ be a Noetherian domain containing $\mbbQ$, $A$ an $\A{2}$-fibration over $R$. Then $A = R^{[2]}$ if and only if there exist a tuple of $R$-LNDs $(D_1, D_2)$ of $A$ with slices and an element $V$ in $A$ such that $D_1(V) = 1$ and $D_2(V)=0$.
	\end{enumerate}

\end{rem}

From our discussions we observe the following criterion for an $\A{3}$-fibration over a Noetherian ring containing $\mathbb{Q}$ to be trivial.

\begin{cor} \label{Cor_A3-fib}
Let $R$ be a Noetherian ring containing $\mbbQ$, $A$ an $\A{3}$-fibration over $R$. Suppose that there exist a tuple of $R$-LNDs $(D_1, D_2)$ of $A$ with slices and an element $V$ in $A$ such that $D_1(V) = 1$ and $D_2(V)=0$. Then, $A = C^{[2]}$ for some $\A{1}$-fibration $C$ over $R$. Further, if $A$ has another $R$-LND $D_3$ with slice and an element $W \in A$ such that $D_2(W) =1 $ and $D_3(V) = D_3(W) =0$, then $A = R^{[3]}$.
\end{cor}
\begin{proof}
Since $D_1(V) =1$, we have $A = \Ker{D_1}[V] = \Ker{D_1}^{[1]}$. As the Zariski's cancellation problem has affirmative answer in dimension three over fields containing $\mbbQ$ (follows from \cite{Miyanishi_cyllinder}, \cite{Fujita_Zariski-Problem}, and \cite{K_K2Form}), from the proof of Lemma \ref{Lem_A=C[W]_implies_W-Residual} we see that $\Ker{D_1}$ is finitely generated and flat over $R$ and $\Ker{D_1} \otimes_R k(P) = k(P)^{[2]}$ for all $P \in \Spec{R}$, i.e., $\Ker{D_1}$ is an $\A{2}$-fibration over $R$. Now, since $D_2$ is an $R$-LND of $A = \Ker{D_1} [V]$ with slice satisfying $D_2(V) =0$, it induces an $R$-LND $\ol{D}_2$ of $A/VA = \Ker{D_1}$ having a slice, and hence, by Corollary \ref{Cor_Structure_A2-Fib_Equivalences}, we have $\Ker{D_1} = C^{[1]}$, i.e., $A = C^{[2]}$ where $C$ is an $\A{1}$-fibration over $R$.

\medskip

We now assume that $A$ has another $R$-LND $D_3$ with slice and an element $W \in A$ such that $D_2(W) =1 $ and $D_3(V) = D_3(W) =0$. Since $A = \Ker{D_1}[V]$ and $\Ker{D_1}$ is an $\A{2}$-fibration over $R$, we see that $V$ is a residual variable of $A$, and therefore, by Theorem \ref{DD_fib}, $A$ is an $\A{2}$-fibration over $R[V]$. Again, since $R[V] \subset \Ker{D_2} \subset \Ker{D_2}[W] = A$, from  Lemma \ref{Lem_A=C[W]_implies_W-Residual} it follows that $\Ker{D_2}$ is an $\A{1}$-fibration over $R[V]$. Now, note that as $V, W \in \Ker{D_3}$ we can see $D_3$ as an $R[V]$-LND of $A = \Ker{D_2}[W]$ with a slice, and therefore, the corresponding $R[V]$-LND of $A/WA = \Ker{D_2}$ also has a slice. This shows that $\Ker{D_2} = R[V]^{[1]}$, i.e., $A = R[V,W]^{[1]} = R^{[3]}$.
\end{proof}

\section{Examples} \label{Sec_Examples} We conclude the article by quoting three examples. The first one establishes the necessity of the condition ``$A$ is stably polynomial over $R$'' in Theorem \ref{Kahoui-A2-fib_trivial_criterion}.

\begin{ex} \label{Ex_Yanik}
Let $B$ be a non-trivial $\A{1}$-fibration over a Noetherian domain $R$ containing $\mathbb{Q}$ (may refer to \cite[Example 1]{Yanik}). Set $A : = B[X] = B^{[1]}$. It is easy to see that $A$ is a non-trivial $\A{2}$-fibration over $R$. Let $B \otimes_R K = K[Y]$ for some $Y \in B$, and therefore, $A \otimes_R K = K[X,Y]$. Let $D$ be the restriction of the partial derivative $\partial_X: A \otimes_R K \longrightarrow A \otimes_R K$ on $A$, i.e., $D = {\partial_X}|_A$. It is easy to see that $D$ is an $R$-LND of $A$ and $D(X) =1$. However, from Theorem \ref{Hamann}, it follows that $A$ is not a stably polynomial algebra over $R$.
\end{ex}	

In view of Corollary \ref{Cor_Structure_A2-Fib_Equivalences} we quote below a class of examples of $\A{2}$-fibrations which are stably polynomial and can not be written as $\A{1}$-fibrations over a polynomial algebras, and therefore, by Theorem \ref{Kahoui-A2-fib_trivial_criterion}, do not possess fixed point free LNDs.
\begin{ex} \label{Ex_Hochster}
Let $A$ be a non-trivial $\A{2}$-fibration over a Noetherian domain $R$ containing $\mathbb{Q}$ such that $A$ is stably polynomial over $R$, e.g., the examples of Raynaud in \cite{Raynaud} (also see \cite{Suslin_Mennicke-symbols}) and Hochster in \cite{Hochster_nonunique-coeff} (for both the examples one may also refer to \cite{Essen-Rossum_counterexamlpes}, \cite{Freudenburg_Rxyz-slice} and \cite[p. 272 and p. 282]{Freudenburg-BookNew}). If possible, let $V \in A$ be such that $A$ is an $\A{1}$-fibration over $R[V]$. By Theorem \ref{Asanuma_struct-fib-th} we see that $A$ is an $R[V]$ subalgebra of $R[V]^{[n]}$ for some $n \in \mathbb{N}$, and therefore, by \cite[Theorem 57, p. 186]{Matsumura_Algebra}, we get $\Omega_R(A) = \Omega_{R[V]}(A) \ \oplus \ \Omega_R(R[V]) \otimes_{R[V]} A = \Omega_{R[V]}(A) \ \oplus \ A$. Since $A$ is stably polynomial over $R$, it can be seen that $\Omega_R(A)$ is a stably free $A$-module (see \cite[Lemma 2.1]{DD_residual}), and therefore, $\Omega_{R[V]}(A)$ is a stably free $A$-module. Now, from Theorem \ref{Asanuma_struct-fib-th} it directly follows that $A$ is a stably polynomial algebra over $R[V]$, and hence, by Theorem \ref{Hamann}, $A = R[V]^{[1]}$ which is a contradiction. This proves that there does not exist $V \in A$ such that $A$ is an $\A{1}$-fibration over $R[V]$.
\end{ex}	

Lastly, again in view of Corollary \ref{Cor_Structure_A2-Fib_Equivalences}, we discuss an example of a non-trivial $\A{2}$-fibration by Asanuma-Bhatwadekar (\cite[Example 3.12]{Asan-Bhatw_Struct-A2-fib}), which is a non-stably polynomial algebra having no fixed point free LND, and can be written as an $\A{1}$-fibration over a polynomial algebra.
	\begin{ex} \label{Ex_AB}
		Set $T := \mbbC[X] = \mbbC^{[1]}$ and $R := \mbbC[X^2, X^3]$. Let $T[V,W] = T^{[2]}$ and $A = R[V, W+ XV^2 W^2] + X^2T[V,W]$. One may check that $A$ is a non-stably polynomial $\A{2}$-fibration over $R$ and $V$ is a residual variable of $A$ over $R$; i.e., $A$ is a non-trivial $\A{1}$-fibration over $R[V]$.
		
		\medskip
		Asanuma-Bhatwadekar, in \cite{Asan-Bhatw_Struct-A2-fib}, established that $A$ can not be written as $A_1 \otimes_R A_2$ where $A_1, A_2$ are $\A{1}$-fibrations over $R$; and therefore, from ``(I) $\Longleftrightarrow$ (III)'' of Corollary \ref{Cor_Structure_A2-Fib_Equivalences}, we see that $A$ does not possess any fixed point free $R$-LND. However, it is easy to see that $A$ has a non-fixed point free $R$-LND with kernel $R[V]$.
	\end{ex}

	\section*{Acknowledgment} The authors thank Neena Gupta for helpful discussions and many fruitful suggestions, and thank Amartya K. Dutta for helping to identify a gap in an earlier draft of this article. The first author sincerely thanks S.M. Bhatwadekar for motivating discussions on a few related topics at CAAG-2017, IISER Pune.

\bibliographystyle{apalike}
\normalem
	\bibliography{reference}
\end{document}